\documentclass[11pt]{article}
\usepackage{amsfonts}
\usepackage{latexsym,amsmath,amsthm,amssymb,color}
\usepackage{mathrsfs,wasysym,graphicx,lscape}

\begin{document}
\newtheorem{theorem}{\indent Theorem}[section]
\newtheorem{proposition}[theorem]{\indent Proposition}
\newtheorem{definition}[theorem]{\indent Definition}
\newtheorem{lemma}[theorem]{\indent Lemma}
\newtheorem{remark}[theorem]{\indent Remark}
\newtheorem{corollary}[theorem]{\indent Corollary}

\begin{center}
    {\large \bf Optimal bilinear control of nonlinear Schr\"{o}dinger equations with singular potentials}
\vspace{0.5cm}\\{Binhua Feng$^{*}$, Dun Zhao, Pengyu Chen}\\
{\small School of Mathematics and Statistics, Lanzhou University\\
Lanzhou, 730000, P.R. China }\\
\end{center}

\renewcommand{\theequation}{\arabic{section}.\arabic{equation}}
\numberwithin{equation}{section}
\footnote[0]{\hspace*{-7.4mm}
E-mail: binhuaf@163.com(Binhua Feng)\\
$^*$Corresponding author\\
This work is supported by the Program for the Fundamental Research
Funds for the Central Universities, NSFC Grants 11031003 and
11171028, and the Program for NCET. }

\renewcommand{\baselinestretch}{1.3}
\large\normalsize
\begin{abstract}
In this paper, we consider an optimal bilinear control problem for
the nonlinear Schr\"{o}dinger equations with singular potentials. We
show well-posedness of the problem and existence of an optimal
control. In addition, the first order optimality system is
rigorously derived. Our results generalize the ones in \cite{Sp} in
several aspects.

{\bf Keywords:} Optimal bilinear control; Nonlinear Schr\"{o}dinger
equation; Optimal condition
\end{abstract}
\section{Introduction}
This paper is devoted to study an optimal bilinear control problem
for the following nonlinear Schr\"{o}dinger equation(NLS):
\begin{equation}\label{1.1}
\left\{
\begin{array}{l}
iu_{t}+\Delta u+ \lambda |u|^{2\sigma}u+
\phi ( t)V(x)u=0,\text{ }%
(t,x)\in [0,\infty )\times \mathbb{R}^{N}, \\
u(0,x) = u_0 (x),%
\end{array}%
\right.
\end{equation}%
where $u_0 \in H^{1}(\mathbb{R}^{N})$, $\phi (t)$ denotes the
control parameter and $V(x)$ is a given potential. The problem of
quantum control via external potentials $\phi (t)V(x)$, has
attracted a great deal of attention from physicians, see
\cite{BVR,HRB,Ho}. From the mathematical point of view, quantum
control problems are a specific example of optimal control problems,
see \cite{Co}, which consist in minimizing a cost functional
depending on the solution of a state equation (here, equation
\eqref{1.1}) and to characterize the minimum of the functional by an
optimality condition.

The mathematical research for optimal bilinear control of systems
governed by partial differential equations has a long history, see
\cite{Fa,Li} for a general overview. However, there are only a few
rigorous mathematical results about optimal bilinear control of
Schr\"{o}dinger equations. Recently, optimal control problems for
linear Schr\"{o}dinger equations have been investigated in
\cite{IK,BK,BS}. Moreover, those results have been tested
numerically in \cite{BS,YKY}. In particular, a mathematical
framework for optimal bilinear control of abstract linear
Schr\"{o}dinger equations was presented in \cite{IK}. In \cite{BK},
the authors considered the optimal bilinear control for the linear
Schr\"{o}dinger equations including coulombian and electric
potentials. For the following NLS of Gross-Pitaevskii type:
\begin{equation}\label{1.1'}
\left\{
\begin{array}{l}
iu_{t}+\Delta u-U(x)u-\lambda |u|^{2\sigma}u-
\phi ( t)V(x)u=0,\text{ }%
(t,x)\in [0,\infty )\times \mathbb{R}^{N}, \\
u(0,x) = u_0 (x),%
\end{array}%
\right.
\end{equation}%
where $\lambda \geq 0$, $U(x)$ is a subquadratic potential,
consequently restricting initial data $u_0\in \Sigma:=\{u\in
H^1(\mathbb{R}^{N}),~ and ~~xu\in L^2(\mathbb{R}^{N})\}$. The
authors in \cite{Sp}
 have presented a novel choice
for the cost term, which is based on the corresponding physical work
performed throughout the control process.
 The proof of the existence of an optimal control relies
heavily on the compact embedding $\Sigma \hookrightarrow
L^2(\mathbb{R}^{N})$. In contrast with \eqref{1.1'}, due to absence
of $U(x)u$ in \eqref{1.1}, we consider \eqref{1.1} in
$H^1(\mathbb{R}^{N})$. Therefore, how to overcome the difficulty
that embedding $H^1(\mathbb{R}^{N}) \hookrightarrow
L^2(\mathbb{R}^{N})$ is not compact, which is of particular
interest,  is one of main technique challenges in this paper.

Borrowing the idea of \cite{Sp}, we now define our optimal control
problem. The
natural candidate for an energy corresponding to \eqref{1.1} is
\begin{equation} \label{1.2}
E(t)=\frac{1}{2}\int_{\mathbb{R}^N}|\nabla
u(t,x)|^2dx-\frac{\lambda}{2\sigma+2
}\int_{\mathbb{R}^N}|u(t,x)|^{2\sigma +2}
dx-\frac{\phi(t)}{2}\int_{\mathbb{R}^N}V(x)|u(t,x)|^2dx.
\end{equation}
Although equation \eqref{1.1} enjoys mass conservation, i.e.,
$\|u(t,\cdot)\|_{L^2}=\|u_0\|_{L^2}$ for all $t\in \mathbb{R}$, the
energy $E(t)$ is not conserved. Indeed, its evolution is given by
\begin{equation} \label{1.3}
\frac{dE(t)}{dt}=-\frac{1}{2}\phi' ( t)\int_{\mathbb{R}^N} V(x)|u(t,x)|^2dx.
\end{equation}
Integrating this equality over the compact interval $[0,T]$, we
obtain
\begin{equation} \label{1.4}
E(T)-E(0)=\frac{1}{2}\int_0^T \phi' ( t)\int_{\mathbb{R}^N}
V(x)|u(t,x)|^2dxdt.
\end{equation}

 For any given $T>0$, we consider $H^1(0,T)$ as the real
vector space of control parameters $\phi $. Set
\begin{equation} \label{3.1}
X(0,T):=L^2((0,T),H_0^1)\cap W^{1,2}((0,T),H^{-1}),
\end{equation}.
For some $M_1>0$ and $M_2>0$, set $B_1:=\{u_0 \in H^1 ~and~\|u_0\|_{H^1}\leq M_1\}$ and $B_2:=\{\phi_0 \in \mathbb{R} ~and~|\phi_0|\leq M_2\}$
\[
\Lambda(0,T):=\{(u,\phi )\in X(0,T)\times H^1(0,T):~ u ~is ~the~
solution ~of~ \eqref{1.1}~ with~u(0)\in B_1~and~\phi(0)\in B_2\}.
\]
Thanks to Lemma 2.3, the set $\Lambda(0,T)$ is not empty. We
consequently define the objective functional $F=F(u,\phi)$ on
$\Lambda(0,T)$ by
\begin{equation} \label{1.5}
F(u,\phi):=\langle u(T,\cdot),Au(T,\cdot)\rangle
_{L^2}^2+\gamma_1\int_0^T(E'(t))^2dt+\gamma_2\int_0^T(\phi
'(t))^2dt,
\end{equation}
where parameters $\gamma_1 \geq 0$ and $\gamma_2>0$, $A :
H^1\rightarrow L^2$ is a bounded linear operator, essentially
self-adjoint on $L^2$ and localizing, i.e., there exists $R>0$, such
that for all  $\psi \in H^1$: $supp_{x\in \mathbb{R}^N}(A\psi
(x))\subseteq B(R)$.



Now, we can define the following minimizing problem:
\begin{equation}\label{3.2}
F_*=\inf_{(u,\phi )\in \Lambda(0,T)}F(u,\phi).
\end{equation}
Firstly, we consider the existence of a minimizer for the above
minimizing problem. This is what the following theorem shows:
\begin{theorem}\label{th1}
Assume $0<\sigma <\frac{2}{N-2}$ if $\lambda <0$, or $0<\sigma
<\frac{2}{N}$ if $\lambda >0$. Let $V\in L^p+L^{\infty}$ for some
$p\geq 1$, $p>N/2$. Then, for any $T>0$, $M_1>0$, $M_2>0$,
$\gamma_1\geq0$ and $\gamma_2>0$, the optimal control problem
\eqref{3.2} has a minimizer $(u_*,\phi_*)\in\Lambda(0,T)$.
\end{theorem}
{\textbf {Remarks.}} (1) In contrast with the results in \cite{Sp},
our results hold for unbounded potential $V$, both focusing and
defocusing nonlinearities. A typical example satisfying our
assumption on $V$ is $\frac{1}{|x|^\alpha}$ for some $0<\alpha <1$.

(2) Since the embedding $H^1(\mathbb{R}^{N})\hookrightarrow
L^2(\mathbb{R}^{N})$ is not compact, the method in \cite{Sp} fails
to work in our situation. We can derive the compactness of a
minimizing sequence by Propositions 1.1.2 and 1.3.14 in
\cite{Ca2003}.

Thanks to global well-posedness of equation \eqref{1.1}, for any
given initial data $u_0\in H^1$, we can define a mapping by
\[
u:~H^1(0,T)\rightarrow X(0,T):~~~\phi \mapsto u(\phi ).
\]
Using this mapping we introduce the unconstrained functional
\[
\mathcal{F}:H^1(0,T)\rightarrow \mathbb{R},~~~\phi \mapsto
\mathcal{F}(\phi):=F(u(\phi),\phi).
\]
In the following theorem, we investigate the differentiability of
unconstrained functional $\mathcal{F}$, and consequently obtain the first order optimality system.
\begin{theorem}\label{th2}
Let $N\leq 3$, $u_0 \in H^2$, $V, \nabla V\in L^p +L^{\infty}$ and
$V\in L^{2p}$ for some $p\geq 2$. Assume $\frac{1}{2}\leq \sigma
<\frac{2}{N-2}$ if $\lambda <0$, or $\frac{1}{2}\leq \sigma
<\frac{2}{N}$ if $\lambda >0$. Then the functional
$\mathcal{F}(\phi)$ is G\^{a}teaux differentiable and
\begin{equation}\label{e3}
\mathcal{F}'(\phi)=Re \int_{\mathbb{R}^N }
\bar{\varphi}(t,x)V(x)u(t,x)dx-2\frac{d}{dt}(\phi'(t)(\gamma_2+\gamma_1
\omega^2(t))),
\end{equation}
in the sense of distributions, where $\omega(t)=\int_{\mathbb{R}^N
}V(x)|u(t,x)|^2dx$ and $\varphi \in C([0,T],L^2)$ is the solution of
the adjoint equation \eqref{eq1}.
\end{theorem}
{\textbf {Remarks.}} (1) Under the assumptions on $u_0 $ and $V$, it
follows from Lemma 2.5 that the solution $u \in L^\infty((0,T),H^2)$
of \eqref{1.1}. Hence, we deduce from the inequality
$\|Vu\|_{L^2}\leq \|V\|_{L^p}\|u\|_{L^{\frac{2p}{p-2}}}$ and
$\varphi \in C([0,T],L^2)$ that the right hand side of \eqref{e3} is
well-defined.

(2) Because control potential $V$ is unbounded, we cannot follow the
method in \cite{Sp} to obtain sufficiently high regularity of $u$,
the solution of the NLS equation \eqref{1.1}. We resume the idea due
to T.Kato, (see, \cite{Ca2003}), based on the general idea for
Schr\"{o}dinger equations, that two space derivative cost the same
as one time derivative.

(3) In contrast with the assumption $\sigma \in \mathbb{N}$ in
\cite{Sp}, our results follow for $\frac{1}{2}\leq \sigma
<\frac{2}{N-2}$ if $\lambda <0$, or $\frac{1}{2}\leq \sigma
<\frac{2}{N}$ if $\lambda >0$.

As an immediate corollary of Theorem 1.2, we derive the precise
characterization for the critical points $\phi_*$ of functional
$\mathcal{F}$. The proof is the same as that of Corollary 4.8 in
\cite{Sp}, so we omit it.
\begin{corollary}\label{th3}
Let $u_*$ be the solution of \eqref{1.1} with control $\phi_*$, and
$\varphi_*$ be the solution of corresponding adjoint equation
\eqref{eq1}. Then $\phi_*\in C^2(0,T)$ is a classical solution of
the following ordinary differential equation
\begin{equation}
\frac{d}{dt}(\phi'_*(t)(\gamma_2+\gamma_1\omega_*^2(t)))=\frac{1}{2}Re
\int_{\mathbb{R}^N }\bar{\varphi_*}(t,x)V(x)u_*(t,x)dx.
\end{equation}
subject to the initial data $\phi_*(0)=\phi_0$ and $\phi'_*(T)=0$.
\end{corollary}

This paper is organized as follows: in Section 2, we will collect
some preliminaries such as compactness results, global existence and
regularity of \eqref{1.1}. In section 3, we will show Theorem 1.1.
In section 4, we firstly analyze well-posedness of the adjoint
equation. Next, the Lipschitz continuity of solution $u=u(\phi)$
with respect to control parameter $\phi$ is obtained. Finally, we
give the proof of Theorem 1.2. Some of the steps of the proof follow
\cite{Sp}, to avoid repetitions we will mainly focus on the
differences with respect to \cite{Sp}.

{\textbf {Notation.}} Throughout this paper, we use the following
notation. $C> 0$ will stand for a constant that may different from
line to line when it does not cause any confusion. Since we
exclusively deal with $\mathbb{R}^N$, we often use the
abbreviations $ L^{r}=L^{r}(\mathbb{R}^{N})$,
$H^s=H^s(\mathbb{R}^N)$.
  Given any interval $I\subset
\mathbb{R}$, the norms of mixed spaces
$L^q(I,L^{r}(\mathbb{R}^{N}))$ and $L^q(I,H^s(\mathbb{R}^N))$ are
denoted by $\|\cdot\|_{L^q(I,L^{r})}$ and $\|\cdot\|_{L^q(I,H^s)}$
respectively. We denote by $U(t):=e^{it\triangle }$ the free
Schr\"{o}dinger propagator, which is isometric on $H^{s}$ for every
$s\geq 0$, see \cite{Ca2003}. We recall that a pair of exponents $(q,r)$ is
Schr\"{o}dinger-admissible if
$\frac{2}{q}=N(\frac{1}{2}-\frac{1}{r})$ and $2 \leq r \leq
\frac{2N}{N-2}$, ($ 2\leq r\leq \infty \, \text{ if } \, N=1$;
$2\leq r <\infty \, \text{ if } \,  N=2$).

\section{Preliminaries}
In this section, we recall some useful results. First, we recall the
following two compactness lemmas which is vital in our paper, see
\cite{Ca2003} for detailed presentation.
\begin{lemma}\cite{Ca2003}
Let $X\hookrightarrow Y$ be two Banach spaces, $I$ be a bounded,
open interval of $\mathbb{R}$, and $(u_n)_{n\in\mathbb{N}}$ be a
bounded sequence in $C(\bar{I},Y)$. Assume that $u_n(t)\in X$ for
all $(n,t)\in \mathbb{N}\times I$ and that
$sup\{\|u_n(t)\|_X,(n,t)\in \mathbb{N}\times I\}=K<\infty$. Assume
further that $u_n$ is uniformly equicontinuous in $Y$. If $X $ is
reflexive, then there exist a function $u\in C(\bar{I},Y)$ which is
weakly continuous $\bar{I}\rightarrow X$ and some subsequence
$(u_{n_k})_{k\in\mathbb{N}}$ such that for every $t\in \bar{I}$,
$u_{n_k}(t)\rightharpoonup u(t)$ in $X$ as $k\rightarrow \infty$.
\end{lemma}
\begin{lemma}\cite{Ca2003}
Let $I$ be a bounded interval of $\mathbb{R}$, and
$(u_n)_{n\in\mathbb{N}}$ be a bounded sequence of
$L^\infty(I,H^1_0)\cap W^{1,\infty}(I,H^{-1})$. Then, there exist
$u\in L^\infty(I,H^1_0)\cap W^{1,\infty}(I,H^{-1})$ and some
subsequence $(u_{n_k})_{k\in\mathbb{N}}$ such that for every $t\in
\bar{I}$, $u_{n_k}(t)\rightharpoonup u(t)$ in $H^1_0$ as
$k\rightarrow \infty$.
\end{lemma}
In the following lemma, we establish some existence results of
equation \eqref{1.1}.
\begin{lemma}
Let $u_0\in H^1$ and $V\in L^p+L^{\infty}$ for some $p\geq 1$,
$p>N/2$. Assume $0<\sigma <\frac{2}{N-2}$ if $\lambda <0$, or
$0<\sigma <\frac{2}{N}$ if $\lambda >0$. For any given $T>0$, $\phi
\in H^1(0,T)$, there exists a unique mild solution $u \in
C([0,T],H^1)$ of problem \eqref{1.1}.
\end{lemma}

\begin{proof}
When $\phi$ is a constant, the author in \cite{Ca2003} showed that
the solution of \eqref{1.1} is local well-posedness. For our case,
since $\phi \in H^1(0,T)\hookrightarrow L^\infty(0,T)$, we only need
to take the $L^\infty$ norm of $\phi$ when the term $\phi Vu$ has to
be estimated in some norms. Keeping this in mind and applying the
method in \cite{Ca2003}, one can show the local well-posedness of
\eqref{1.1}. Hence, in order to prove this lemma, it suffices to show
\begin{equation}\label{100}
\|u(t)\|_{H^1} \leq C(T,\|u_0\|_{H^1},\phi).
\end{equation}
Indeed, we deduce from \eqref{1.3} and mass conservation that
\[
\|E'\|_{L^2(0,T)}\leq
C\|\phi'\|_{L^2(0,T)}(\|V_1\|_{L^p}\|u\|_{L^{\frac{2p}{p-1}}}^2+
\|V_2\|_{L^\infty}\|u_0\|_{L^2}^2).
\]
This implies
\begin{align*}
E(t)=&E(0)+\int_0^t E'(s)ds\leq
E(0)+\left(T\int_0^T(E'(s))^2ds\right)^{1/2}\nonumber\\\leq &
E(0)+CT^{1/2}\|\phi'\|_{L^2(0,T)}(\|V_1\|_{L^p}\|u\|_{L^{\frac{2p}{p-1}}}^2+
\|V_2\|_{L^\infty}\|u_0\|_{L^2}^2).
\end{align*}
When $\lambda \leq 0$, it follows from \eqref{1.2} that
\begin{align}\label{101}
\|\nabla u(t)\|_{L^2 }^2\leq &CE(t)+C\|\phi\|_{L^\infty(0,T)}\int
V|u|^2dx \nonumber
\\\leq & CE(0)+CT^{1/2}\|\phi'\|_{L^2(0,T)}(\|V_1\|_{L^p}\|u\|_{L^{\frac{2p}{p-1}}}^2+
\|V_2\|_{L^\infty}\|u_0\|_{L^2}^2) \nonumber
\\&+C\|\phi\|_{L^\infty(0,T)}(\|V_1\|_{L^p}\|u\|_{L^{\frac{2p}{p-1}}}^2+
\|V_2\|_{L^\infty}\|u_0\|_{L^2}^2),
\end{align}
which, together with the embedding $H^1\hookrightarrow
L^{\frac{2p}{p-1}}$ and Young's inequality with $\varepsilon$,
implies \eqref{100}.

When $\lambda > 0$, by the same argument as above, we have
\begin{align}
\|\nabla u(t)\|_{L^2 }^2 \leq &
CE(0)+CT^{1/2}\|\phi'\|_{L^2(0,T)}(\|V_1\|_{L^p}\|u\|_{L^{\frac{2p}{p-1}}}^2+
\|V_2\|_{L^\infty}\|u_0\|_{L^2}^2) \nonumber
\\&+C\|\phi\|_{L^\infty(0,T)}(\|V_1\|_{L^p}\|u\|_{L^{\frac{2p}{p-1}}}^2+
\|V_2\|_{L^\infty}\|u_0\|_{L^2}^2)+C\|u\|_{L^{2\sigma +2}}^{2\sigma
+2},
\end{align}
It follows from Gagliardo-Nirenberg's inequality
 that
\begin{equation}\label{l0}
\|u\|_{L^{2\sigma +2}}^{2\sigma +2}\leq C\|u\|_{H^1}^{N\sigma
}\|u\|_{L^2}^{2\sigma +2-N\sigma}.
\end{equation}
Since $N\sigma<2$, \eqref{100} follows from Young's inequality with
$\varepsilon$.
\end{proof}
\begin{lemma}\cite{Ca2003}
Let $J\ni 0$ be a bounded interval, $(\gamma,\rho)$ be an admissible
pair and consider $f \in L^{\infty}(J,L^2)$ such that $f_t \in
L^{\gamma^\prime}(J,L^{\rho^\prime})$. If
\[
v(t)=i\int_0^t U(t-s)f(s)ds~~~for~all~t\in J,
\]
then $v \in L^{\infty}(J,H^2)\cap C^1(J,L^2)\cap W^{1,a}(J,L^b)$ for
every admissible pair $(a,b) $ and
\[
\|\Delta v\|_{ L^{\infty}(J,L^2)}\leq \|f\|_{ L^{\infty}(J,L^2)}
+\|f(0)\|_{ L^2}+C \|f\|_{L^{\gamma^\prime}(J,L^{\rho^\prime})},
\]
where $C$ is independent of $J$ and $f$.
\end{lemma}

\begin{lemma}
Let $u_0\in H^2$, $\phi \in H^1(0,T)$ and $V, \nabla V \in
L^p+L^{\infty}$ for some $p\geq 2$, $p>N/2$. Assume $0<\sigma
<\frac{2}{N-2}$ if $\lambda <0$ or $0<\sigma <\frac{2}{N}$ if
$\lambda >0$. Then the mild solution of \eqref{1.1} satisfies $u\in
L^\infty((0,T),H^2)$.
\end{lemma}
This lemma can be proved by applying Remarks 5.3.3 and 5.3.5 in
\cite{Ca2003}. When $0<\sigma <\frac{2}{N}$, for this lemma, it
suffices to require $V\in L^p+L^{\infty}$ for some $p\geq 1$,
$p>N/2$, see Remark 5.3.5 in \cite{Ca2003}.

\section{Existence of Minimizers}

Our goal in this section is to prove Theorem 1.1.

{\textbf {Proof of Theorem 1.1.}} The proof proceeds in three steps.

Step 1. Estimates of $(u_n,\phi_n)_{n\in\mathbb{N}}$. Let $\phi \in
H^1(0,T)$, there exists a unique mild solution $u\in C([0,T],H^1)$
of \eqref{1.1} by Lemma 2.3. Hence, the set $\Lambda(0,T)$ is
nonempty, and there exists a minimizing sequence
$(u_n,\phi_n)_{n\in\mathbb{N}}$ such that
\[
\lim_{n\rightarrow \infty}F(u_n,\phi_n)=F_*.
\]
We deduce from $\gamma_2>0$ that there exists a constant $C$ such
that for every $n\in\mathbb{N}$
\[
\int_0^T(\phi'_n(t))^2dt\leq C<+\infty.
\]
By using the embedding $H^1(0,T)\hookrightarrow C[0,T]$ and
$\phi_n(0)\in B_2$, we have
\[
\phi_n(t)=\phi_n(0)+\int_0^t\phi'_n(s)ds\leq
M_2+\left(T\int_0^T(\phi'_n(s))^2ds\right)^{1/2}<+\infty.
\]
This implies the sequence $(\phi_n)_{n\in\mathbb{N}}$ is bounded in
$L^\infty(0,T)$, so is in $H^1(0,T)$. Thus,
there exist a subsequence, which we still denote by
$(\phi_n)_{n\in\mathbb{N}}$, and $\phi_*\in H^1(0,T)$ such that
\begin{equation}\label{n0}
\phi_n\rightharpoonup \phi_*~~in ~~H^1(0,T)~~and ~~\phi_n
\rightarrow \phi_*~~in ~~L^2(0,T)~~as~~n\rightarrow \infty.
\end{equation}
On the other hand, we deduce from \eqref{1.3} and mass conservation
that
\[
\|E'_n\|_{L^2(0,T)}\leq
C\|\phi'_n\|_{L^2(0,T)}\|V\|_{L^\infty}\|u_0\|_{L^2}^2.
\]
 Using the same argument as
Lemma 2.3 and $u_n(0)\in B_2$, we derive
\begin{equation}\label{35}
\|u_n\|_{L^\infty ((0,T),H^1)}\leq C.
\end{equation}
Combining this estimate and the fact that $u_n$ is the solution of
\eqref{1.1}, we have
\begin{equation}\label{35'}
\|(u_n)_t\|_{L^{\infty }((0,T),H^{-1})}\leq C.
\end{equation}


Step 2. Passage to the limit. By applying \eqref{35}, \eqref{35'},
and Lemma 2.2, we deduce that there exist $u_*\in L^\infty
((0,T),H^1)\cap W^{1,\infty }((0,T),H^{-1})$ and a subsequence,
still denoted by $(u_n)_{n\in\mathbb{N}}$, such that, for all
$t\in[0,T]$,
\begin{equation}\label{31}
u_n(t)\rightharpoonup u_*(t)~~in ~H^1~as~n\rightarrow \infty.
\end{equation}

From the embedding $W^{1,\infty }((0,T),H^{-1})\hookrightarrow
C^{0,1}([0,T],H^{-1})$ (see \cite{Ca2003}, Remark 1.3.11) and the
inequality $\|u\|^2_{L^2}\leq \|u\|_{H^{1}}\|u\|_{H^{-1}}$, we
derive for every $u\in L^\infty ((0,T),H^1)\cap W^{1,\infty
}((0,T),H^{-1})$
\begin{equation}\label{l1}
\|u(t)-u(s)\|_{L^2}\leq C|t-s|^{\frac{1}{2}},~~for~all~t,s\in(0,T).
\end{equation}
Next, we note that for all $z_1,z_2 \in \mathbb{C}$, it holds
\begin{equation}\label{l2}
||z_1|^{2\sigma}z_1-|z_2|^{2\sigma} z_2|\leq
C(|z_1|^{2\sigma}+|z_2|^{2\sigma})|z_1-z_2|.
\end{equation}
It follows from \eqref{l0}, \eqref{35}, \eqref{l1}, \eqref{l2},
H\"{o}lder's inequality that
\begin{align}\label{l3}
\||u_n(t)|^{2\sigma} u_n(t)-|u_n(s)|^{2\sigma}
u_n(s)\|_{L^{r^\prime}}\leq & C( \|
u_n(t)\|_{L^r}^{2\sigma}+\|u_n(s)\|_{L^r}^{2\sigma})\|u_n(t)-u_n(s)\|_{L^r}\nonumber\\\leq&
C\|u_n(t)-u_n(s)\|_{L^2}^a\leq C|t-s|^{\frac{a}{2}},
\end{align}
where $r=2\sigma +2$ and $a=1-N(\frac{1}{2}-\frac{1}{2\sigma +2})$.
This implies $(|u_n|^{2\sigma} u_n)_{n\in\mathbb{N}}$ is a bounded
sequence in $C^{0,\frac{a}{2}}([0,T],L^{r^\prime})$. Therefore, we
deduce from Lemma 2.1 that there exist a subsequence, still denoted
by $(|u_n|^{2\sigma} u_n)_{n\in\mathbb{N}}$, and $f\in
C^{0,\frac{a}{2}}([0,T],L^{r^\prime})$ such that, for all
$t\in[0,T]$,
\begin{equation}\label{32}
|u_n(t)|^{2\sigma} u_n(t)\rightharpoonup f(t)~~in
~L^{r^\prime}~as~n\rightarrow \infty.
\end{equation}
On the other hand, it follows from $(u_n,\phi_n)\in \Lambda(0,T)$
that for every $\omega \in C_c^\infty (\mathbb{R}^N)$ and for every $\eta \in
\mathcal{D}(0,T)$,
\[
\int_0^T[-\langle iu_n,\omega \rangle_{H^{-1},H^1_0}\eta'(t)+
\langle \Delta u_n+|u_n|^{2\sigma} u_n+\phi _n(t)Vu_n,\omega
\rangle_{H^{-1},H^1_0}\eta(t)]dt=0.
\]
Applying \eqref{n0}, \eqref{31}, \eqref{32}, and the dominated
convergence theorem, we deduce easily that
\[
\int_0^T[-\langle iu_*,\omega \rangle_{H^{-1},H^1_0}\eta'(t)+\langle \Delta u_*+f+\phi _*(t)Vu_*,\omega \rangle_{H^{-1},H^1_0}\eta(t)]dt=0.
\]
This implies that $u_*$ satisfies
\begin{equation}\label{36}
i\frac{d}{dt}u_* + \Delta u_*+f+\phi _*(t)Vu_*=0~~for~a.e.~t \in
[0,T].
\end{equation}
We next show $|u_*(t,x)|^{2\sigma}u_*(t,x)=f(t,x)$ for a.e.
$(t,x)\in [0,T]\times \mathbb{R}^N$.  It suffices to show that for any given
$t\in [0,T]$ 
\begin{equation}\label{38}
\int_{\mathbb{R}^N} |u_*(t,x)|^{2\sigma}(x)u_*(t,x)\varphi
(x)dx=\int_{\mathbb{R}^N}f(t,x)\varphi (x)dx~~for ~any~\varphi \in
C_c^\infty (\mathbb{R}^N).
\end{equation}
 Let us prove \eqref{38}
by contradiction. If not, there exists $\varphi_0 \in C_c^\infty
(\mathbb{R}^N)$ such that
\begin{equation}\label{39}
\int_{\mathbb{R}^N} |u_*(t,x)|^{2\sigma}(x)u_*(t,x)\varphi _0(x)dx\neq
\int_{\mathbb{R}^N}f(t,x)\varphi_0 (x)dx.
\end{equation}
It follows from \eqref{32} that
\begin{equation}\label{40}
\int_{\mathbb{R}^N} |u_n(t,x)|^{2\sigma}(x)u_n(t,x)\varphi _0(x)dx\rightarrow
\int_{\mathbb{R}^N}f(t,x)
 \varphi_0 (x)dx ~~as~n\rightarrow \infty.
\end{equation}
On the other hand, we deduce from \eqref{31} that there exists a
subsequence, still denoted by $(u_n(t))_{n\in\mathbb{N}}$ such that
$u_n(t)\rightarrow u_*(t)$ in $L^{2\sigma +2}_{loc}(\mathbb{R}^N)$ and
$|u_n(t)|^{2\sigma}\rightarrow |u_*(t)|^{2\sigma}$ in $L^{\frac{2\sigma
+2}{2\sigma}}_{loc}(\mathbb{R}^N)$. Combining this, \eqref{35} and
\eqref{31}, we derive
\begin{align}\label{41}
&\left|\int_{\mathbb{R}^N} |u_n(t,x)|^{2\sigma}u_n(t,x)\varphi
_0(x)dx-\int_{\mathbb{R}^N}|u_*(t,x)|^{2\sigma}u_*(t,x)
 \varphi_0 (x)dx\right|\nonumber\\
  \leq &\|u_n(t)\|_{L^{2\sigma +2}}^{2\sigma}\|u_n(t)-u_*(t)\|_{L^{2\sigma +2}(\Omega)}
  \|\varphi_0\|_{L^{2\sigma +2}}+
\|u_*(t)\|_{L^{2\sigma +2}}\|u_n(t)-u_*(t)\|_{L^{\frac{2\sigma
+2}{2\sigma}}(\Omega)}\|\varphi_0\|_{L^{2\sigma +2}}\nonumber\\
&\xrightarrow {n\rightarrow \infty} 0 ,
\end{align}
where $\Omega $ is the compact support of $\varphi_0$. This is a
contradiction with \eqref{39} and \eqref{40}.

In summary, $u_*\in L^\infty ((0,T),H^1)\cap W^{1,\infty
}((0,T),H^{-1})$ and satisfies
\[
i\frac{d}{dt}u_* + \Delta u_*+\lambda |u_*|^{2\sigma}u_*+\phi
_*(t)Vu_*=0, ~~for~a.e.~t\in [0,T].
\]
By using the classical argument based on Strichartz's estimate, we
can obtain the uniqueness of the weak solution $u_*$ of \eqref{1.1}.
Arguing as the proof of Theorem 3.3.9 in \cite{Ca2003}, it follows
that $u_*$ is indeed a mild solution of \eqref{1.1} and $u_*\in
C((0,T),H^1)\cap C^1((0,T),H^{-1})$.

Step 3. Conclusion. In order to conclude that the pair
$(u_*,\phi_*)\in \Lambda(0,T)$ is indeed a minimizer of optimal
control problem \eqref{3.2}, we need only show
\begin{equation}\label{310}
F_*=\lim_{n\rightarrow \infty}F(u_n,\phi_n)\geq F(u_*,\phi_*).
\end{equation}
Indeed, in view of the assumption on operator $A$, there exists
$R>0$, such that for every $n \in \mathbb{N}$, $supp_{x\in
\mathbb{R}^3}(Au(T,x))\subseteq B(R)$. Therefore, we deduce from
$u_n(T)\rightarrow u_*(T)$ in $L^2_{loc}$ and
$Au_n(T)\rightharpoonup Au_*(T)$ in $L^2$ that
\begin{align}\label{311}
&|\langle u_n(T), Au_n(T)\rangle-\langle u_*(T), Au_*(T)\rangle | \nonumber\\
\leq &|\langle u_n(T)-u_*(T), Au_n(T)\rangle | + |\langle u_*(T), A(
u_n(T)-u_*(T))\rangle |\rightarrow 0~~as~n\rightarrow \infty.
\end{align}
The same argument as Lemma 2.5 in \cite{Sp}, we have
\begin{equation}\label{312}
\liminf_{n\rightarrow
\infty}\int_0^T(\phi_n'(t))^2\omega_n^2(t)dt\geq
\int_0^T(\phi'_*(t))^2\omega_*^2(t)dt,
\end{equation}
where
\[
\omega_n(t)=\int_{\mathbb{R}^3 }V(x)|u_n(t,x)|^2dx,
~~\omega_*(t)=\int_{\mathbb{R}^3 }V(x)|u_*(t,x)|^2dx.
\]
It follows from the weak lower semicontinuity of the norm that
\begin{equation}\label{313}
\liminf_{n\rightarrow \infty}\int_0^T(\phi_n'(t))^2dt\geq
\int_0^T(\phi'_*(t))^2dt.
\end{equation}
Collecting \eqref{311}-\eqref{313}, we derive \eqref{310}. This
completes the proof.

\section{Rigorous characterization of a minimizer}
In order to obtain a rigorous characterization of a minimizer
$(u_*,\phi_*)\in \Lambda(0,T)$, we need to derive the first order
optimality conditions for our optimal control problem \eqref{3.2}.
For this aim, we firstly formally calculate the derivative of the
objective functional $F(u,\phi )$ and analyze the resulting adjoint
problem in the next subsection.

\subsection{Derivation and analysis of the adjoint equation.}

To begin with, we rewrite equation \eqref{1.1} in a more abstract
form, i.e.,
\begin{equation}\label{402}
P(u,\phi )=iu_{t}+\Delta u+\lambda  |u|^{2\sigma}u+ \phi (
t)V(x)u=0.
\end{equation}
Thus, formally computation yields
\[
\partial_uP(u,\phi)\varphi =i\varphi_t+\Delta \varphi +\phi(t)V(x)\varphi
+\lambda (\sigma +1)|u|^{2\sigma}\varphi +\lambda
|u|^{2\sigma-2}u^2\bar{\varphi},
\]
where $\varphi \in L^2$. Similarly, we have
\[
\partial_\phi P(u,\phi) =V(x)u.
\]
The analogue argument as Section 3.1 in \cite{Sp}, we can derive the
following adjoint equation:
\begin{equation}\label{eq1}
\left\{
\begin{array}{l}
i\varphi_t+\Delta \varphi +\phi(t)V(x)\varphi +\lambda (\sigma
+1)|u|^{2\sigma}\varphi +\lambda |u|^{2\sigma-2}u^2\bar{\varphi}
 =\frac{\delta F(u,\phi)}{\delta u(t)}, \\
\varphi(T)=i\frac{\delta F(u,\phi)}{\delta u(T)},%
\end{array}%
\right.
\end{equation}
where $\frac{\delta F(u,\phi)}{\delta u(t)}$ and $\frac{\delta
F(u,\phi)}{\delta u(T)}$ denote the first variation of $F(u,\phi)$
with respect to $u(t)$ and $u(T)$ respectively. By straightforward
computations, we have
\begin{align}\label{e1}
\frac{\delta F(u,\phi)}{\delta u(t)}=&4\gamma
_1(\phi'(t))^2(\int_{\mathbb{R}^N}V(x)|u(t,x)|^2dx)V(x)u(t,x)
\nonumber\\
=& 4\gamma _1(\phi'(t))^2\omega(t)V(x)u(t,x),
\end{align}
 and
\begin{equation}\label{e2}
\frac{\delta F(u,\phi)}{\delta u(T)}=4\langle
u(T),Au(T)\rangle_{L^2} Au(T).
\end{equation}
Thus, equation \eqref{eq1} defines a Cauchy problem for $\varphi$
with data $\varphi(T)\in L^2$, one can solve \eqref{eq1} backwards
in time.

In the following proposition, we will analyze the existence of
solutions to \eqref{eq1}.
\begin{proposition}
Let $N\leq 3$, $u_0\in H^2$ and $V, \nabla V \in L^p+L^{\infty}$ for
some $p\geq 2$. Assume $0<\sigma <\frac{2}{N-2}$ if $\lambda <0$ or
$0<\sigma <\frac{2}{N}$ if $\lambda >0$. Then, for every $T>0$,
equation \eqref{eq1} admits a unique mild solution $\varphi \in
C([0,T],L^2)$.
\end{proposition}
\begin{proof}
Under our assumptions on $V$, $u_0$, and $A$, we deduce from
$H^2\hookrightarrow L^\infty$ and Lemma 2.5 that $|u|^{2\sigma}$,
$|u|^{2\sigma-2}u^2 \in L^\infty$, $\frac{\delta F(u,\phi)}{\delta
u(t)} \in L^1((0,T),L^{\frac{2p}{p+1}})$, $\frac{\delta
F(u,\phi)}{\delta u(T)}\in L^2$. Since $V$ is an unbounded
potential, it cannot be treated as a perturbation. Applying
consequently Theorem 4.6.4 and Corollary 4.6.5 in \cite{Ca2003}, we
can obtain the local well-posedness. The global existence can be
derived by the classical argument for Schr\"{o}dinger equations and
Gronwall's inequality.
\end{proof}

\subsection{Lipschitz continuity with respect to the control.}
This subsection is devoted to derive the solution of \eqref{1.1}
depends Lipschitz continuously on the control parameter $\phi$,
which is vital for investigating the differentiability of
unconstrained functional $\mathcal{F}$. To begin with, we study the
continuous dependence of the solutions $u=u(\phi)$ with respect to
the control parameter $\phi$. Our result is as follows.

\begin{proposition}
Let $N\leq 3$, $V, \nabla V\in L^p +L^{\infty}$ and $V\in L^{2p}$
for some $p\geq 2$.  Assume $0<\sigma <\frac{2}{N-2}$ if $\lambda
<0$ or $0<\sigma <\frac{2}{N}$ if $\lambda >0$. Let $u,\tilde{u}\in
L^\infty ((0,T),H^2)$ be two mild solutions of \eqref{1.1} with the
same initial data $u_0\in H^2$, corresponding to control parameters
$\phi,\tilde{\phi}\in H^1(0,T)$ respectively. Given a constant
$M>0$, if
\[
\|\phi\|_{H^1(0,T)},~\|\tilde{\phi}\|_{H^1(0,T)},~
\|u(t)\|_{H^2},~\|\tilde{u}(t)\|_{H^2}\leq M,
\]
then, there exist $\tau =\tau(M)>0$ and a constant $C=C(M)$ such
that
\begin{equation}\label{46}
\|u-\tilde{u}\|_{L^\infty(I_t,H^2)}\leq
C(\|u(t)-\tilde{u}(t)\|_{H^2}+\|\phi-\tilde{\phi}\|_{L^2(I_t)}),
\end{equation}
where $I_t:=[t,t+\tau]\cap[0,T]$. In particular, the solution $
u(\phi)$ depends continuously on control parameter $\phi \in
H^1(0,T)$.
\end{proposition}
\begin{proof}
To simplify notation, let us assume $t+\tau \leq T$. Applying Lemmas
2.3 and 2.4, there is a $\tau
>0$ depending only on $M$, such that $u|_{I_t}$ is a fixed point of
the operator
\[
\Phi(u):=U(\cdot-t)u(t)+i\int_t^\cdot U(\cdot-s)(\lambda
|u(s)|^{2\sigma}u(s)+ \phi(s)Vu(s))ds,
\]
which maps the Banach space
\[
Y=\{u\in L^\infty (I_t,H^2), ~~\|u\|_{L^\infty (I_t,H^2)}\leq 2M \}
\]
into itself. The same holds for $\tilde{u}$, we consequently derive
\begin{align}\label{f1}
\tilde{u}(s)-&u(s)=U(s-t)(\tilde{u}(t)-u(t))\nonumber\\
+ &i\int_t^{s} U(s-r)(\lambda
(|\tilde{u}|^{2\sigma}\tilde{u}-|u|^{2\sigma}u)+
V(\tilde{u}\tilde{\phi}-u\phi))(r)dr
\end{align}
where $s\in [t,t+\tau]$. In the following, we set $r=2\sigma +2$ and
$\rho =\frac{2p}{p-1}$, taking $q$ and $\gamma$ such that $(q,r)$
and $(\gamma,\rho)$ are two admissible pairs. Applying Strichartz's
estimate to \eqref{f1}, the embedding theorems $\tilde{\phi}, \phi
\in H^1(0,T)\hookrightarrow L^\infty (0,T)$ and $\tilde{u}(t,\cdot),
u(t,\cdot) \in H^2(\mathbb{R}^{N})\hookrightarrow
L^\infty(\mathbb{R}^{N})$ when $N\leq 3$, H\"{o}lder's inequality,
\eqref{l2}, we derive
\begin{align}\label{f2}
&\|\tilde{u}-u\|_{L^\infty (I_t,L^2)}\nonumber\\\leq&
C\|\tilde{u}(t)-u(t)\|_{L^2}+
C\||\tilde{u}|^{2\sigma}\tilde{u}-|u|^{2\sigma}u\|_{L^{q^\prime}(I_t,L^{r^\prime})}+
C\|V(\tilde{u}\tilde{\phi}-u\phi))\|_{L^{\gamma^\prime}(I_t,L^{\rho^\prime})}\nonumber\\
\leq & C\|\tilde{u}(t)-u(t)\|_{L^2}+C\tau^{\frac{q-q^\prime}{qq^\prime}}\left(\|\tilde{u}\|^{2\sigma}_{L^\infty(I_t,L^r)}+C\|u\|^{2\sigma}_{L^\infty(I_t,L^r)}\right)
\|\tilde{u}-u\|_{L^q(I_t,L^r)}\nonumber\\
&+C\|V\|_{L^p}\|\tilde{\phi}\|_{L^\infty}\|\tilde{u}-u\|_{L^{\gamma^\prime}(I_t,L^{\rho})}
+\|V\|_{L^p}\|\tilde{\phi}-\phi\|_{L^{\gamma^\prime}(I_t)}\|u\|_{L^{\infty}(I_t,L^{\rho})}\nonumber\\
\leq
&C\|\tilde{u}(t)-u(t)\|_{L^2}+C\tau^{\frac{1}{q^\prime}}\|\tilde{u}-u\|_{L^\infty(I_t,H^2)}
+C\tau^{\frac{1}{\gamma^\prime}}\|\tilde{u}-u\|_{L^\infty(I_t,H^2)}+C\|\tilde{\phi}-\phi\|_{H^1(I_t)}.
\end{align}
Set $f_1(t)=\lambda (
|\tilde{u}|^{2\sigma}\tilde{u}-|u|^{2\sigma}u)(t)$ and $f_2(t)=
V(\tilde{u}\tilde{\phi}-u\phi)(t)$, we deduce from Lemma 2.4 that
\begin{align}\label{f3}
\|\Delta (\tilde{u}-u)\|_{L^\infty(I_t,L^2)}\leq & \|\Delta
(\tilde{u}(t)-u(t))\|_{L^2}+C\|f_1(t)+f_2(t)\|_{L^2}+\|f_1+f_2
\|_{L^\infty(I_t,L^2)}\nonumber\\&+C\|(f_1)_t\|_{L^{q^\prime}(I_t,L^{r^\prime})}
+C\|(f_2)_t\|_{L^{\gamma^\prime}(I_t,L^{\rho^\prime})}.
\end{align}
Let us estimate these terms.
By the similar argument as \eqref{f2}, we obtain
\begin{align}\label{f4}
\|f_1(t)+f_2(t)\|_{L^2}\leq& C\|(|\tilde{u}|^{2\sigma}+|u|^{2\sigma})(t)(\tilde{u}-u)(t)\|_{L^2}+\|V\tilde{\phi}(t)(\tilde{u}-u)(t)\|_{L^2}+
\|Vu(t)(\tilde{\phi}-\phi)(t)\|_{L^2}\nonumber\\
\leq & C\|\tilde{u}(t)-u(t)\|_{L^2}+\tilde{\phi}(t)\|V\|_{L^p}\|(\tilde{u}-u)(t)\|_{L^{\frac{2p}{p-2}}}+
|(\tilde{\phi}-\phi)(t)|\|Vu(t)\|_{L^2}\nonumber\\
\leq &C\|\tilde{u}(t)-u(t)\|_{L^2}+C\|\tilde{u}(t)-u(t)\|_{H^2}+C
\|\tilde{\phi}-\phi\|_{H^1(I_t)};
\end{align}
When $2<p<\infty$, $2<\frac{2p}{p-2}<\infty$, we deduce from
interpolation inequality and Young's inequality with $\varepsilon$
that
\begin{align}\label{f5}
&\|f_1+f_2\|_{L^\infty (I_t,L^2)}\nonumber\\ \leq&
C\|(|\tilde{u}|^{2\sigma}+|u|^{2\sigma})(\tilde{u}-u)\|_{L^\infty
(I_t,L^2)}+\|V\tilde{\phi}(\tilde{u}-u)\|_{L^\infty (I_t,L^2)}+
\|Vu(\tilde{\phi}-\phi)\|_{L^\infty (I_t,L^2)}\nonumber\\
\leq & C\|\tilde{u}-u\|_{L^\infty
(I_t,L^2)}+C\|V\|_{L^p}\|\tilde{u}-u\|_{L^\infty(I_t,L^{\frac{2p}{p-2}})}+
\|\tilde{\phi}-\phi\|_{H^1(I_t)}\|V\|_{L^p}\|u\|_{L^\infty (I_t,H^2)}\nonumber\\
\leq & C\|\tilde{u}-u\|_{L^\infty (I_t,L^2)}+\varepsilon
\|\tilde{u}-u\|_{L^\infty (I_t,H^2)}+
C\|\tilde{\phi}-\phi\|_{H^1(I_t)}.
\end{align}
When $p=2$, by the similar argument as above, we have
\begin{align}
\|V\tilde{\phi}(\tilde{u}-u)\|_{L^\infty (I_t,L^2)}\leq &
\|V\|_{L^4}\|\tilde{\phi}(\tilde{u}-u)\|_{L^\infty (I_t,L^4)}\nonumber\\
\leq & C\|\tilde{u}-u\|_{L^\infty (I_t,L^2)}+\varepsilon
\|\tilde{u}-u\|_{L^\infty (I_t,H^2)};
\end{align}
When $p=\infty$,
\begin{equation}
\|V\tilde{\phi}(\tilde{u}-u)\|_{L^\infty (I_t,L^2)}\leq
\|V\|_{L^\infty}\|\tilde{\phi}(\tilde{u}-u)\|_{L^\infty (I_t,L^2)}.
\end{equation}

 After some fundamental computations, by the similar argument as
\eqref{f2}, we obtain
\begin{align}\label{f6}
\|(f_1)_t\|_{L^{q^\prime} (I_t,L^{r^\prime})}\leq &
C\||\tilde{u}|^{2\sigma}\tilde{u}_t-|u|^{2\sigma}u_t\|_{L^{q^\prime}
(I_t,L^{r^\prime})}+\||\tilde{u}|^{2\sigma-2}\bar{\tilde{u}}_t
\tilde{u}^2-|u|^{2\sigma-2}\bar{u}_tu^2\|_{L^{q^\prime}
(I_t,L^{r^\prime})}\nonumber\\\leq & C
\|\tilde{u}_t-u_t\|_{L^{q^\prime}(I_t,L^2)}+\||\tilde{u}|^{2\sigma-2}
\tilde{u}^2(\bar{\tilde{u}}_t-\bar{u}_t)\|_{L^{q^\prime}
(I_t,L^{r^\prime})}\nonumber\\ &+\|(|\tilde{u}|^{2\sigma-2}
\tilde{u}^2-|u|^{2\sigma-2}u^2)\bar{u}_t\|_{L^{q^\prime}
(I_t,L^{r^\prime})}\nonumber\\\leq & C
\|\tilde{u}_t-u_t\|_{L^{q^\prime}(I_t,L^2)}+C\|(\tilde{u}-u)(|\tilde{u}|^{2\sigma-1}
+|u|^{2\sigma-1})\bar{u}_t\|_{L^{q^\prime}
(I_t,L^{r^\prime})}\nonumber\\\leq & C
\|\tilde{u}_t-u_t\|_{L^{q^\prime}(I_t,L^2)}+\tau^{\frac{1}{q^\prime}}\|\tilde{u}-u\|_{L^\infty(I_t,H^2)},
\end{align}
where
\begin{align}\label{f7}
& \|\tilde{u}_t-u_t\|_{L^{q^\prime} (I_t,L^{2})}\nonumber\\\leq &
C\|\Delta (\tilde{u}-u)\|_{L^{q^\prime} (I_t,L^{2})}+C\|
V(\tilde{u}\tilde{\phi}-u\phi)\|_{L^{q^\prime} (I_t,L^{2})}+ C\|
|\tilde{u}|^{2\sigma}\tilde{u}-|u|^{2\sigma}u\|_{L^{q^\prime}
(I_t,L^{2})}\nonumber\\\leq & C
\tau^{\frac{1}{q^\prime}}\|\tilde{u}-u\|_{L^\infty(I_t,H^2)}+
C\|\tilde{\phi}-\phi\|_{H^1(I_t)}+\tau^{\frac{1}{q^\prime}}\|\tilde{u}-u\|_{L^\infty(I_t,L^2)}.
\end{align}
Similarly,
\begin{align}\label{f8}
\|(V(\tilde{u}\tilde{\phi}-u\phi)))_t\|_{L^{\gamma^\prime}(I_t,L^{\rho^\prime})}\leq
 & \|V(\tilde{\phi}'-\phi')\tilde{u})\|_{L^{\gamma^\prime}(I_t,L^{\rho^\prime})}+
 \|V\phi'(\tilde{u}-u)\|_{L^{\gamma^\prime}(I_t,L^{\rho^\prime})} \nonumber\\&+
 \|V(\tilde{\phi}-\phi)\tilde{u}_t)\|_{L^{\gamma^\prime}(I_t,L^{\rho^\prime})}+\|V\phi(\tilde{u}_t-u_t)\|_{L^{\gamma^\prime}(I_t,L^{\rho^\prime})}
\nonumber\\\leq
 &C \|\tilde{\phi}-\phi\|_{H^1(I_t)}+\tau^{\frac{2-\gamma ^\prime}{2\gamma ^\prime}}
 \|\phi'\|_{L^2}\|\tilde{u}-u\|_{L^\infty(I_t,H^2)}\nonumber\\&+
 \|V\|_{L^{2p}}\|\tilde{\phi}-\phi\|_{H^1(I_t)}+
 \|V\|_{L^{2p}}\|\tilde{u}_t-u_t\|_{L^{\gamma^\prime}(I_t,L^2)}.
\end{align}
Notice that the estimates \eqref{f2}-\eqref{f8} hold for $V \in
L^\infty$. Combining \eqref{f2}-\eqref{f8}, using the equivalent
norm of $H^2$, i.e., $\|\cdot\|_{H^2}=\|\cdot\|_{L^2}+\|\Delta
\cdot\|_{L^2}$, we obtain
\begin{align}\label{f9}
\|\tilde{u}-u\|_{L^\infty (I_t,H^2)}\leq
&C\|\tilde{u}(t)-u(t)\|_{H^2}+
C\tau^{\frac{1}{q^\prime}}\|\tilde{u}-u\|_{L^\infty(I_t,H^2)}+
C\tau^{\frac{1}{\gamma^\prime}}\|\tilde{u}-u\|_{L^\infty(I_t,H^2)} \nonumber\\
&+ C\|\tilde{\phi}-\phi\|_{H^1(I_t)}+\tau^{\frac{2-\gamma
^\prime}{2\gamma ^\prime}}
\|\tilde{u}-u\|_{L^\infty(I_t,H^2)}+\varepsilon
\|\tilde{u}-u\|_{L^\infty (I_t,H^2)}.
\end{align}
Therefore,
 the estimate \eqref{46} holds by choosing $\tau$
and $\varepsilon$ sufficiently small. Due to $\tilde{u}(0)=u(0)$, we
deduce from continuity argument and \eqref{46} that the mapping
$\phi \rightarrow u(\phi)$ is continuous with respect to $\phi \in
H^1(0,T)$.
\end{proof}

 We are now
in the position to show Lipschitz continuity of solution $u(\phi)$
with respect to $\phi \in H^1(0,T)$. With the estimate \eqref{46} at
hand, the proof is analogue to that of Proposition 4.5 in \cite{Sp},
so we omit it.
\begin{proposition}
Let $N\leq 3$, $V, \nabla V\in L^p +L^{\infty}$ and $V\in L^{2p}$
for some $p\geq 2$.  Assume $0<\sigma <\frac{2}{N-2}$ if $\lambda
<0$ or $0<\sigma <\frac{2}{N}$ if $\lambda >0$. Let
 $\phi \in H^1(0,T)$, and $u=u(\phi)\in
L^\infty ((0,T),H^2)$ be the solution of \eqref{1.1}. Given
$\delta_\phi \in H^1(0,T)$ with $\delta_\phi(0)=0$, for every
$\varepsilon \in [-1,1]$, let $\tilde{u}=u(\phi+\epsilon
\delta_\phi) $ be the solution of \eqref{1.1} with control
$\phi+\epsilon \delta_\phi$ and the same initial data as $u(\phi)$.
Then, there exists a constant $C>0$ such that
\[
\|\tilde{u}-u\|_{L^\infty((0,T),H^2)}\leq
C\|\tilde{\phi}-\phi\|_{H^1(0,T)}=C|\varepsilon|\|\delta_\phi\|_{H^1(0,T)}.
\]
In other words, the mapping $\phi \mapsto u(\phi)$ is Lipschitz
continuous with respect to $\phi$ for each fixed direction
$\delta_\phi$.
\end{proposition}

{\textbf {Proof of Theorem 1.2.}} In view of definition of
G\^{a}teaux derivative, let $u=u(\phi)$, $\tilde{u}=u(\tilde{\phi})$
with $\tilde{\phi}=\phi +\varepsilon\delta_\phi$, we compute
\[
\mathcal{F}(\tilde{\phi})-\mathcal{F}(\phi)=\mathcal{J}_1+\mathcal{J}_2+
\mathcal{J}_3,
\]
where
\[
\mathcal{J}_1:=\langle \tilde{u}(T),A\tilde{u}(T)\rangle^2-\langle
u(T),Au(T)\rangle^2,
\]
\[
\mathcal{J}_2:=\gamma_2\int_0^T
\big[(\tilde{\phi}'(t))^2-(\phi'(t))^2\big]dt,
\]
and
\[
\mathcal{J}_3:=\gamma_1\int_0^T
(\tilde{\phi}'(t))^2\left(\int_{\mathbb{R}^N}
V(x)|\tilde{\phi}(t,x)|^2\right)^2dt-\gamma_1\int_0^T
(\phi'(t))^2\left(\int_{\mathbb{R}^N} V(x)|\phi(t,x)|^2\right)^2dt.
\]
Because we have obtained Proposition 4.3, $\tilde{u}, u \in
L^\infty((0,T),H^2)\hookrightarrow L^\infty ((0,T)\times
\mathbb{R}^N)$, ones can prove along the lines of Theorem 4.6 in
\cite{Sp}, so we omit it.


\begin{thebibliography}{11}
{\small




\vspace{-0.3 cm}
\bibitem{BVR} V. Bulatov, B. E. Vugmeister, H. Rabitz, Nonadiabatic
control of Bose-Einstein condensation in optical traps, Phys. Rev.
A, 60(1999) 4875-4881.

\vspace{-0.3 cm}

\bibitem{HRB} U. Hohenester, P. K. Rekdal, A. Borzi, J. Schmiedmayer,
Optimal quantum control of Bose Einstein condensates in magnetic
microtraps, Phys. Rev. A, 75(2007) 023602-023613.

 \vspace{-0.3 cm}

\bibitem{Ho} M. Holthaus, Toward coherent control of Bose-Einstein
condensate in a double well, Phys. Rev. A, 64(2001) 011601-011608.

\vspace{-0.3 cm}

\bibitem{Co} J.-M. Coron, Control and Nonlinearity. Mathematical Surveys
and Monographs, vol. 136, American Mathematical Society, 2007.
\vspace{-0.3 cm}
\bibitem{Fa} H. Fattorini, Infinite dimensional optimization and control theory, Cambridge University
Press, 1999.

\vspace{-0.3 cm}

\bibitem{Li} J.L. Lions, Optimal control of systems governed by
 partial differential equations,
Springer Verlag, 1971.
\vspace{-0.3 cm}

\bibitem{IK} K. Ito, K. Kunisch, Optimal bilinear control of an abstract
Schr\"{o}dinger equation, SIAM J. Control Optim., 46(2007) 274-287.

\vspace{-0.3 cm}

\bibitem{BK} L. Baudouin, O. Kavian, J.P. Puel, Regularity for a Schr\"{o}dinger equation with
singular potentials and application to bilinear optimal control, J.
Diff. Equ., 216(2005) 188-222.
\vspace{-0.3 cm}

\bibitem{BS} L. Baudouin, J. Salomon, Constructive solution of a bilinear optimal control problem
for a Schr\"{o}dinger equation, Systems Control Lett., 57(2008)
454-464.

\vspace{-0.3cm}

\bibitem{YKY} B. Yildiz, O. Kilicoglu, G. Yagubov, Optimal control problem for nonstationary
Schr\"{o}dinger equation, Num. Methods Partial Diff. Equ., 25(2009)
1195-1203.


\vspace{-0.3 cm}

\bibitem{Sp} M. Hinterm{\"u}ller, D. Marahrens, P.A. Markowich, C. Sparber,
Optimal bilinear control of Gross-Pitaevskii equations, Arxiv
preprint arXiv:1202.2306.


\vspace{-0.3 cm}

\bibitem{Ca2003} T. Cazenave, Semilinear Schr\"{o}dinger equations, Courant Lecture Notes in Mathematics vol. 10,
New York University, Courant Institute of Mathematical Sciences, New
York; American Mathematical Society, Providence, RI, 2003.






 }
\end{thebibliography}
\end{document}